\documentclass{amsart}

\usepackage{amsmath,mathtools}
\usepackage{enumerate}
\usepackage{nicefrac}

\newtheorem{thm}{Theorem}[section]
\newtheorem{cor}[thm]{Corollary}
\newtheorem{lemma}[thm]{Lemma}
\newtheorem{prop}[thm]{Proposition}

\theoremstyle{remark}
\newtheorem{remark}[thm]{Remark}

\title{Small totally $p$-adic algebraic numbers}
\author{Lukas Pottmeyer}
\address{Universit\"at Duisburg-Essen, Fakult\"at f\"ur Mathematik, D-45117 Essen}
\email{lukas.pottmeyer@uni-due.de}
\date{\today}

\begin{document}

\begin{abstract}
The purpose of this note is to give a short and elementary proof of the fact, that the absolute logarithmic Weil-height is bounded from below by a positive constant for all totally $p$-adic numbers which are neither zero nor a root of unity. The proof is based on an idea of C. Petsche and gives the best known lower bounds in this setting. These bounds differ from the truth by a term of less than $\nicefrac{\log(3)}{p}$.
\end{abstract}

\thanks{I thank W. Zudilin for initiating this work by asking for a $p$-adic analogue of the proof in \cite{HS}, and for providing feedback on an early version of this document. Moreover, I thank the anonymous referee for providing a simplification of the proof of Proposition \ref{propxpxp}, and for other valuable comments. Finally, I thank Clayton Petsche for pointing out an error in one of the citations.}

\maketitle

\section{Summary of results}

In this paper we denote by $\mathbb{Q}^{tp}$ the field of totally $p$-adic numbers. This means that $\mathbb{Q}^{tp}$ consists of all algebraic numbers $\alpha$ such that the prime $p$ splits completely in $\mathbb{Q}(\alpha)$. An equivalent definition is, that $\mathbb{Q}^{tp}$ is the largest Galois extension of $\mathbb{Q}$ which can be embedded into the $p$-adic numbers $\mathbb{Q}_p$.

For any algebraic number $\alpha$, the (absolute logarithmic Weil-)height of $\alpha$ is 
\[
h(\alpha)=\frac{1}{[\mathbb{Q}(\alpha):\mathbb{Q}]}\sum_{v\in M_{\mathbb{Q}(\alpha)}} d_v \max\{\log(\vert \alpha \vert_v),0\},
\]
where $M_{\mathbb{Q}(\alpha)}$ is a full set of non-trivial absolute values on $\mathbb{Q}(\alpha)$ extending the standard (archimedean and non-archimedean) absolute values on $\mathbb{Q}$, and $d_v$ is the local degree $[\mathbb{Q}(\alpha)_v : \mathbb{Q}_p]$. Actually, $\mathbb{Q}(\alpha)$ can be replaced by any number field containing $\alpha$, without changing the value of $h(\alpha)$. We will freely use $h(\alpha)=h(\alpha^{-1})$, $h(\alpha^k)=k h(\alpha)$, $h(\alpha \cdot \beta)\leq h(\alpha)+h(\beta)$ and $h(\alpha-1)\leq h(\alpha)+\log(2)$, where $k\in\mathbb{N}$ and $\beta$ is any algebraic number. Moreover we will use that  $h(\alpha-1)=h(\alpha)+\log(2)$ if and only if $\alpha=-1$. For proofs of these and other properties of $h$ we refer to Section 1.5 of \cite{BG}. 

Given a finite set of primes $S$ then Bombieri and Zannier \cite{BZ01} have shown that 
\begin{equation}\label{eq:bz}
\frac{1}{2}\sum_{p\in S} \frac{\log(p)}{p+1} \quad \leq \quad \liminf_{\alpha \in \cap_{p \in S} \mathbb{Q}^{tp}} h(\alpha) \quad \leq \quad \sum_{p \in S} \frac{\log(p)}{p-1}.
\end{equation}
The lower bound is valid also for an infinite set $S$, and was strengthened by Fili and Petsche \cite{FP15} to $\frac{1}{2}\sum_{p\in S} \frac{p\log(p)}{p^2 -1}$. In particular, for $S=\{p\}$ the inequalities in \eqref{eq:bz} imply that the height of a totally $p$-adic number is either zero or bounded from below by a positive constant less or equal to $\frac{\log(p)}{p-1}$. Restricting to algebraic \emph{integers} and $p\neq 2$, this latter fact can be deduced as follows:

\medskip

\begin{quote}
Let $p$ be an odd prime and $\alpha\in\mathbb{Q}^{tp}$ an algebraic integer such that $\alpha^p -\alpha\neq 0$. Then any of the $[\mathbb{Q}(\alpha):\mathbb{Q}]$ extensions $v\mid p$ satisfies either $\vert \alpha \vert_v=1$ or $\vert \alpha \vert_v \leq \frac{1}{p}$. By Fermat's little theorem we get
\begin{align*}
\log(p) &\leq \frac{1}{[\mathbb{Q}(\alpha):\mathbb{Q}]} \sum_{v\mid p} \log(\vert (\alpha^p -\alpha)^{-1} \vert_v) \\ &\leq h((\alpha^p -\alpha)^{-1}) =h(\alpha\cdot (\alpha^{p-1}-1)) \leq h(\alpha) + h(\alpha^{p-1} -1) \\ & \leq h(\alpha) + (p-1)h(\alpha) +\log(2) = p h(\alpha)+\log(2).
\end{align*}
It follows $h(\alpha)\geq \frac{\log(p/2)}{p}$.
\end{quote}

\medskip

Since $h$ vanishes precisely at $0$ and roots of unity, this little argument also proves the well known fact that the $(p-1)$th roots of unity are the only roots of unity in $\mathbb{Q}^{tp}$. 
Note that for algebraic integers $\alpha\notin\{0,\pm1\}$ which are totally real, H\"ohn and Skoruppa \cite{HS} used a similar proof to achieve Schinzel's sharp lower bound $h(\alpha)\geq \frac{1}{2}\log\left(\frac{1+\sqrt{5}}{2}\right)$.

A mild extension of the above proof allows a slightly more careful statement. Therefore, for a rational prime $p$ we define
\begin{align*}
u_p &:= \inf\{ h(\alpha) \vert \alpha \text{ an algebraic unit } \in \mathbb{Q}^{tp} \setminus \{\text{roots of } 1\} \} \\
i_p &:= \inf\{ h(\alpha) \vert  \alpha \text{ an algebraic integer } \in (\mathbb{Q}^{tp})^* \setminus \{\text{roots of } 1\} \} \\
n_p &:= \inf\{ h(\alpha) \vert \alpha \in (\mathbb{Q}^{tp})^* \setminus \{\text{roots of } 1\} \}
\end{align*}

Obviously, we have $u_p\geq i_p \geq n_p$.

\begin{thm}\label{mainthm}
Let $p$ be an odd prime. The following bounds hold true:
\begin{align*}
\frac{\log(p/2)}{p-1} < &u_p < \frac{\log(p+\nicefrac{1}{p^{p-2}})}{p-1}\\
\frac{\log(p/2)}{p} < &i_p \leq \frac{\log(p)}{p} \\
\frac{\log(p/2)}{p+1} < &n_p
\end{align*}
Moreover, $u_p < \frac{\log(p)}{p-1}$ whenever $p$ is not a Fermat-prime.
\end{thm}

The lower bound for $u_p$ has been proved by Petsche \cite{Pe}. Indeed, in Section \ref{sec:lo} we extend his proof to get the lower bounds for $i_p$ and $n_p$. Dubickas and Mossinghoff \cite{DM05} have provided slightly better lower bounds for $u_p$. The upper bounds are proved in Section \ref{sec:upupip}.
In Section \ref{sec:lo23} we will separately handle the primes $2$ and $3$, and we will prove the best known lower bounds in these two cases. Recall, that the Mahler measure of a polynomial $f(x)=a_d(x-\alpha_1)\cdot \ldots\cdot(x-\alpha_d)\in\mathbb{Z}[x]$ is given by
\[
M(f)=\vert a_d \vert \prod_{i=1}^d \max\{\vert \alpha_i \vert,1\},
\]
and that we have $h(\alpha_1)=\frac{1}{d}\log(M(f))$, whenever $f$ is irreducible. If $f$ is not irreducible, then we still can say that some root of $f$ satisfies $h(\alpha) \leq \frac{1}{d}\log(M(f))$. This observation can also be found in \cite{Pe}. If $f=g_1\cdots g_r$ is a decomposition into irreducible factors, then $\frac{\log(M(f))}{\deg(f)} = \frac{1}{\deg(f)} \sum_{i=1}^r \log(M(g_i)) \geq \min_{i} \frac{\log(M(g_i))}{\deg(g_i)}$. 

\begin{thm}\label{thm:23}
The following bounds hold true:

\begin{tabular}{rclcrcl}
$\log(2) <$ & $u_2$ & $\leq \frac{\log(M(x^2 -8x -1))}{2}$ & \quad & $0.294061 <$ & $u_3$ & $\leq \frac{\log(M(x^2-3x-1))}{2}$ \\
$\frac{2}{5}\log(2) <$ & $i_2$ & $\leq \frac{\log(2)}{2}$ & \quad & $0.176437 <$ & $i_3$ & $\leq \frac{\log(2)}{2}$ \\
$\frac{\log(2)}{4} <$ & $n_2$ & & \quad & $0.126026 <$ & $n_3$ &$\leq \frac{\log(3)}{4}$
\end{tabular}
\end{thm}

Dubickas and Mossinghoff \cite{DM05} noticed $n_2 < u_2$, and asked for which other $p$ the strict inequality $n_p < u_p$ is true. The above result shows $n_3< u_3$.

In Section \ref{sec:upnp} we consider upper bounds for $n_p$. Trivially we have $n_p \leq i_p \leq \frac{\log(p)}{p}$. We will prove (as in the case $p=3$) that at least sometimes we have $n_p < \frac{\log(p)}{p}$. Now we generalize the Question of Dubickas and Mossinghoff and ask: For which primes $p$ (if any) do we have $n_p < i_p < u_p$?

\section{Lower bounds: case of odd primes}\label{sec:lo}

\begin{proof}[Proof of the lower bounds in Theorem \ref{mainthm}]  
Let $\alpha \in \mathbb{Q}^{tp}$ be neither $0$ nor a $(p-1)$th root of unity. There are exactly $D=[\mathbb{Q}(\alpha):\mathbb{Q}]$ extensions $v \mid p$ on $\mathbb{Q}(\alpha)$ and all of these satisfy $\vert \alpha \vert_v \in\{p^a \vert a\in\mathbb{Z}\}$.

Let $r$ be the number of extensions $v\mid p$, where $\vert \alpha\vert_v < 1$ (and hence $\leq \frac{1}{p}$), and let $s$ be the number of extensions $v\mid p$, where $\vert \alpha\vert_v >1$  (and hence $\geq p$). Then the number of extensions $v\mid p$, where $\vert \alpha\vert_v =1$, is $D-r-s$. Note, that $s=0$ if $\alpha$ is an algebraic integer, and that $r=s=0$ if $\alpha$ is an algebraic unit. Define 
\[
\delta(\alpha)=\begin{cases} 0 & \text{ if } \alpha \text{ is an algebraic unit } \\
 1 & \text{ if } \alpha \text{ is an algebraic integer and not a unit } \\
 2 & \text{ else } \end{cases}
\]
Since we want to calculate $h(\alpha)$ and we have $h(\alpha)=h(\alpha^{-1})$, we assume without loss of generality $r\geq s$.
Applying the definition of the height, we have
\begin{equation}\label{eq:r}
\frac{r}{D}\cdot \log(p) \leq h(\alpha^{-1})=h(\alpha).
\end{equation}
Since the residue field of $\mathbb{Q}(\alpha)_v$ has exactly $p$ elements, Fermat tells us that $\vert \alpha^{p-1} -1 \vert_v \leq \frac{1}{p}$, whenever $\vert \alpha \vert_v =1$. In particular, we have
\begin{equation}\label{eq:aux}
\frac{D-\delta(\alpha)\cdot r}{D}\cdot \log(p)  \leq \frac{D-r-s}{D}\cdot \log(p) \leq h((\alpha^{p-1}-1)^{-1})< (p-1)h(\alpha)+\log(2)
\end{equation}
and hence
\begin{equation}\label{eq:t}
\frac{D-\delta(\alpha)\cdot r}{D\cdot (p-1)}\cdot \log(p) -\frac{\log(2)}{p-1} < h(\alpha).
\end{equation}
Here, we have used, that $\alpha^{p-1}-1\neq 0$. In \eqref{eq:r} we have a lower bound for $h(\alpha)$ which is linearly increasing in $r$, and in \eqref{eq:t} we have a lower bound which is linearly decreasing in $r$. These two lines intersect at 
\[
r=\frac{D\cdot \log(p/2)}{\log(p)\cdot (p-1+\delta(\alpha))}
\]
which yields the claimed result
\[
h(\alpha)>\frac{\log(p/2)}{p-1+\delta(\alpha)}.
\]
\end{proof}

\begin{remark}
Using auxiliary polynomials in \eqref{eq:aux}, as constructed in \cite{DM05}, one can slightly strengthen these bounds. However, the general form of the bounds (including the miserable $\log(2)$) stays the same. Therefore, we did not include this additional technicality.
\end{remark}

\begin{remark}
For non totally $p$-adic numbers the proof also applies. Let $\alpha$ be neither zero nor a root of unity such that the Galois closure $\mathbb{Q}(\alpha)^{\mathcal{G}}$ of $\mathbb{Q}(\alpha)$ can be embedded into a finite extension of $\mathbb{Q}_p$ with ramification degree $e$ and inertia degree $f$. Then $p$ has exactly $\frac{[\mathbb{Q}(\alpha)^{\mathcal{G}}:\mathbb{Q}]}{ef}$ extensions $v$ to $\mathbb{Q}(\alpha)^{\mathcal{G}}$ and all of these satisfy $\vert \alpha \vert_v \in \{ p^{\nicefrac{a}{e}} \vert a\in \mathbb{Z}\}$. Hence, the same proof as above gives  
\[
h(\alpha)\geq \frac{\log(p/2^e)}{e(p^f -1+\delta(\alpha))}.
\]
This bound, however, is only non-trivial if $p>2^e$, but provides the best known estimate in the unramified setting. A better lower bound in the case where $e$ is large compared to $p$, is given in Theorem 2 of \cite{FP18}.
\end{remark}

We take the opportunity to use the last remark to remove a technical condition in a theorem of A. Galateau \cite{Ga16}. For any set $S$ of primes, let $L_S$ be the compositum of the Hilbert class fields of $\mathbb{Q}(\sqrt{-p})$ for all $p\in S$.

\begin{thm}
Fix an odd prime $q$. There is a subset $S$ of primes of density $\frac{1}{2}$ such that $h(\alpha)> \frac{\log(q/2)}{q^2+1}$ for all $\alpha \in L_S^*\setminus\{ \text{roots of }1\}$.
\end{thm}
\begin{proof}
Let $S$ be the set of primes for which $q$ is inert in $\mathbb{Q}(\sqrt{-p})$. Then, by Chebotarev's density theorem, $S$ has indeed density $\frac{1}{2}$ in the set of all primes. The field $L_S$ is Galois over $\mathbb{Q}$ and can be embedded into the unramified quadratic extension of $\mathbb{Q}_q$. This follows from class field theory, since the principal ideal $q\mathcal{O}_{\mathbb{Q}(\sqrt{-p})}$ splits completely in the Hilbert class field of $\mathbb{Q}(\sqrt{-p})$. Hence, for all $\alpha \in L_S^*$ which are not a root of unity, we have $h(\alpha)> \frac{\log(q/2)}{q^2+1}$. 
\end{proof}

In \cite{Ga16}, the set $S$ has density $\frac{1}{4}$, since the proof requires the additional assumption that all primes in $S$ are congruent $1$ modulo $4$.
 
\section{lower bounds: case of $p\in\{2,3\}$}\label{sec:lo23}

The following lemma is well known and is used implicitly in \cite{DM05} and \cite{BDM}. For the readers convenience, we will present the short proof.

\begin{lemma}\label{lem:maxunitc}
Let $f(x)\in\mathbb{Z}[x]$ be a polynomial, and define $$\Vert f\Vert_\infty
:=\max_{z\in\mathbb{C},\vert z \vert =1} \vert f(z) \vert.$$ Then $h(f(\alpha))\leq \deg(f)h(\alpha) + \log(\Vert f\Vert_\infty )$ for all algebraic numbers $\alpha$.
\end{lemma}
\begin{proof}
Let $\alpha$ be an algebraic number and $v\in M_{\mathbb{Q}(\alpha)}$ be non-archimedean. Then the ultrametric inequality yields $\max\{1,\vert f(\alpha)\vert_v\} \leq \max\{1,\vert \alpha^{\deg(f)} \vert_v\}$. If $y \in \mathbb{C}$ satisfies $\vert y \vert \leq 1$, then the maximum modulus principle tells us $\vert f(y)\vert \leq \Vert f \Vert_{\infty}$. If $y\in \mathbb{C}$ satisfies $\vert y \vert >1$, then $\vert y^{-1} \vert <1$ and by the maximum modulus principle for the polynomial $x^{\deg(f)}f(\frac{1}{x})$ we get
\[
\vert f(y)\vert = \vert y^{\deg(f)} \vert \cdot \vert \frac{1}{y^{\deg(f)}} f(y) \vert \leq \vert y^{\deg(f)} \vert \cdot \max_{\vert z \vert =1}\vert z^{\deg(f)}f(\frac{1}{z}) \vert =\vert y^{\deg(f)} \vert \cdot \Vert f \Vert_\infty.
\]
We have just seen, that for any archimedean $v\in M_{\mathbb{Q}(\alpha)}$ it is $\max\{1,\vert f(\alpha)\vert_v \} \leq \max\{1,\vert \alpha^{\deg(f)} \vert_v \} \cdot \Vert f \Vert_\infty$. Together with the non-archimedean bound from the beginning of the proof, we conclude
\begin{align*}
h(f(\alpha))&=\frac{1}{[\mathbb{Q}(\alpha):\mathbb{Q}]}\sum_{v\in M_{\mathbb{Q}(\alpha)}}d_v \log(\max\{1, \vert f(\alpha) \vert_v\}) \\ &\leq \frac{1}{[\mathbb{Q}(\alpha):\mathbb{Q}]}\sum_{v\in M_{\mathbb{Q}(\alpha)}}d_v \log(\max\{1, \vert \alpha^{\deg(f)} \vert_v\}) +\frac{[\mathbb{Q}(\alpha):\mathbb{Q}]}{[\mathbb{Q}(\alpha):\mathbb{Q}]}\log(\Vert f\Vert_\infty) \\ &= h(\alpha^{\deg(f)})+\log(\Vert f\Vert_\infty)=\deg(f)h(\alpha)+\log(\Vert f\Vert_\infty),
\end{align*}
proving the lemma.
\end{proof}

\begin{lemma}\label{lem:explicitunitc3}
Let $f(x)=(x-4)\cdot(x-1)\cdot (x+2)$. Then, in the notation from above, we have $\Vert f \Vert_\infty = \sqrt{\frac{135}{2}+\frac{189}{4}\cdot\sqrt{7}}=13.8748603\ldots$.
\end{lemma}
\begin{proof}
Take an undetermined element $z_{\theta}=\cos(\theta)+i\cdot \sin(\theta)$ on the unit circle. Then 
\[
\vert f(z_\theta )\vert^2 = (17-8\cdot \cos(\theta))\cdot (2-2\cdot \cos(\theta))\cdot (5+4\cdot \cos(\theta)). 
\]
Considering this as a function in $\cos(\theta)$, yields that this is maximized for $\cos(\theta)=\frac{5}{8} - \frac{3}{8}\cdot \sqrt{7}$. Plugging this back in, gives the claimed maximum. 
\end{proof}

We now prove Theorem \ref{thm:23}. The only new ingredient is essentially the use of the fact that the square of an odd integer is congruent $1$ modulo $8$.

\begin{proof}[Proof of the lower bounds from Theorem 2] 
Let $\alpha$ be totally $p$-adic and neither zero nor $\pm1$. The strategy of the proof is exactly the same as in the proof of Theorem 1. Hence we use the same notation for $r,s,D,\delta(\alpha)$. We first handle the case $p=2$. Then 
\begin{equation}\label{eq:r2}
h(\alpha)\geq \frac{r}{D}\cdot \log(2).
\end{equation}
Let $v\mid 2$ be an absolute value on $\mathbb{Q}(\alpha)$ for which $\vert \alpha \vert_v =1$. Moreover let $\mathfrak{p}_v$ be the corresponding prime ideal. Then, since $v\mid 2$ is unramified, $\alpha^2 -1 \equiv 0 \mod{\mathfrak{p}_v^3}$. In particular, we get
\[
\frac{D-\delta(\alpha)r}{D}\cdot 3\cdot \log(2) \leq h(\alpha^2 -1) < 2h(\alpha)+\log(2),
\]
which yields
\begin{equation}\label{eq:t2}
h(\alpha)> \left(\frac{D-\delta(\alpha)r}{D}\cdot 3-1\right)\cdot \frac{\log(2)}{2}.
\end{equation}
The increasing bound from \eqref{eq:r2} intersects the decreasing bound \eqref{eq:t2} at $r=\frac{2D}{2+3\delta(\alpha)}$. This gives immediately the claimed bounds for $p=2$.

Now let $p=3$. Again we have $h(\alpha)\geq \frac{r}{D}\cdot \log(3)$. Let $v\mid 3$ be such that $\vert \alpha\vert_v =1$. If $\mathfrak{p}_v$ is the corresponding prime ideal, then $\alpha^2 -k\equiv 0 \mod{\mathfrak{p}_v}$ for $k=1,4,-2$. Moreover, for exactly one of these $k$ it is $\alpha^2 -k \equiv 0 \mod{\mathfrak{p}_v^2}$. Using Lemmas \ref{lem:maxunitc} and \ref{lem:explicitunitc3}, we see
\[
\frac{D-\delta(\alpha)r}{D}\cdot 4\cdot \log(3) \leq h((\alpha^2 -4)(\alpha^2 -1)(\alpha^2 +2)) < 6h(\alpha)+\log(13.874861),
\]
which yields
\[
h(\alpha) > \frac{D-\delta(\alpha)r}{D}\cdot 2\cdot \frac{\log(3)}{3}-\frac{\log(13.874861)}{6}.
\]
The same line intersecting method as before concludes the proof.
\end{proof}

\section{upper bounds for $u_p$ and $i_p$}\label{sec:upupip}

\begin{prop}
Let $p$ be an odd prime. Then $u_p < \frac{\log(p+\frac{1}{p^{p-2}})}{p-1}$. If $p$ is not a Fermat-prime, then $u_p < \frac{\log(p)}{p-1}$.
\end{prop}
\begin{proof}
We have to give examples of totally $p$-adic algebraic units of small height. By Hensel's lemma, the polynomial $f_k(x)=x^{p-1} - p\cdot x^k -1$, where $k\in\{1,\ldots,p-2\}$, splits completely over $\mathbb{Q}_p$. By Perron's criterion $f_{p-2}$ has exactly one root outside the unit circle, and hence it is irreducible. We notice 
\[
f_{p-2}(p)=-1<0 \qquad \text{ and } \qquad f_{p-2}(p+\frac{1}{p^{p-2}}) = \frac{(p^{p-1} +1)^{p-2}}{p^{(p-1)(p-2)}} -1 >0.
\] 
Hence, the root of $f_{p-2}$ outside the unit circle lies in the interval $(p,p+\frac{1}{p^{p-2}})$. It follows $0\neq M(f_{p-2})< p+\frac{1}{p^{p-2}}$, which proves the first claim.

Now, let $p$ be a prime such that $p-1=2^n\cdot m$ for integers $n,m$, such that $m>1$ is odd. As before, the polynomial $g(x)=x^m + p\cdot x^{m-1}-1$ has exactly one root outside the unit circle (and hence is irreducible). Since
\[
g(-p+\frac{1}{p^{m-1}}) = \frac{(p^m -1)^{m-1}}{p^{m(m-1)}}-1 < 0 \quad \text{ and }\quad g(-p+1) = (-p+1)^{m-1} -1 >0,
\]
this root lies in the interval $(-p+\frac{1}{p^{m-1}},-p+1)$. It follows
\[
0\neq M(x^m + p\cdot x^{m-1}-1)< p-\frac{1}{p^{m-1}}.
\]
In particular, for any root $\alpha$ of $g$ we have $h(\alpha)<\frac{\log(p-\frac{1}{p^{m-1}})}{m}$. Moreover, $\alpha^{\nicefrac{1}{2^n}}$ is a root of $f_{2^n(m-1)}$, and hence totally $p$-adic, and satisfies $h(\alpha^{\nicefrac{1}{2^n}})=\frac{1}{2^n}\cdot h(\alpha) < \frac{\log(p-\frac{1}{p^{m-1}})}{p-1}<\frac{\log(p)}{p-1}$.
\end{proof}

\begin{prop}\label{propxpxp}
For all primes $p$ it is $i_p \leq \frac{\log(p)}{p}$. For $p=3$ it is $i_3 \leq \frac{\log(2)}{2}$.
\end{prop}
\begin{proof}
By Hensel's lemma the polynomials $x^p -x +p$ and $x^{p-1}+ (p-1)$ split completely over $\mathbb{Q}_p$. The polynomial $x^2 -x +2$ has no root inside the unit circle. If $\alpha$ is a root of $x^p -x +p$ inside the unit circle for an odd prime $p$, then $\vert \alpha \cdot (\alpha^{p-1}-1) \vert =p$. Hence $ 2 \geq\vert \alpha^{p-1}-1\vert \geq p$, which gives a contradiction. Therefore, all roots of $x^p -x +p$ lie outside the unit circle. 
Since $p$ is prime, it follows, that $x^p -x +p$ is irreducible, and any root $\alpha$ satisfies $h(\alpha)=\frac{\log(M(x^p -x +p))}{p}=\frac{\log(p)}{p}$. 
\end{proof}

In order to give all upper bounds for $u_p$ and $i_p$ presented in Theorems \ref{mainthm} and \ref{thm:23}, it only remains to notice, that any root of $x^2 -8x-1$ is totally $2$-adic.

\section{An upper bound for $n_p$}\label{sec:upnp}

We already know that $n_p \leq i_p \leq \frac{\log(p)}{p}$. In this section we show that in some cases $n_p$ is less than $\frac{\log(p)}{p}$. For the rest of this section we fix an odd prime $p$. In order to construct a totally $p$-adic number of height $\frac{\log(p)}{p+1}$, we want to construct a polynomial of leading coefficient $p$, degree $p+1$, and all roots on the unit circle, which splits completely over $\mathbb{Q}_p$. Necessarily such a polynomial must be self-reciprocal. We define
\begin{equation}\label{eq:poly}
f(x)=px^{p+1} - \frac{p+1}{2}x^p - \sum_{i=2}^{p-1} x^i + p\cdot \sum_{i=2}^{\frac{p-1}{2}}(-1)^i\cdot (x^i + x^{p+1-i}) -\frac{p+1}{2}x +p
\end{equation}
Note that $f$ is indeed a self-reciprocal polynomial satisfying
\begin{align*}
f &\equiv -\frac{p+1}{2}\cdot \left(x^p +2x^{p-1} +2\cdot x^{p-2} +\ldots+2x^2 +x\right)\\
  &\equiv  -\frac{p+1}{2}\cdot x \cdot \left(\prod_{i=2}^{p-2}(x-i)\right) \cdot (x+1)^2 \mod{p}.
\end{align*}
By the shape of the Newton polygon of $f$, we see that $f$ has exactly one root in $\mathbb{Q}_p \setminus\mathbb{Z}_p$. Moreover, each of the $p-2$ single roots of the reduction of $f$ lifts to exactly one root of $f$ in $\mathbb{Z}_p$. Hence, $f$ has at least $p-1$ pairwise distinct roots in $\mathbb{Q}_p$. In Lemma \ref{mainlemma} below, we will prove that $\vert f(2p-1) \vert_p \leq \frac{1}{p^3} < \vert f'(2p-1)\vert_p^2$. Therefore, Hensel's lemma in general form (cf. \cite{Lang}, II \S2 Proposition 2) predicts, that $2p-1$ lifts to a root of $f$ in $\mathbb{Z}_p$. This root is necessarily distinct from all the $p-2$ roots corresponding to the simple roots of the reduction of $f$. Hence, $f$ has at least $p$ roots in $\mathbb{Q}_p$ and therefore, since $\deg(f)=p+1$, $f$ splits completely over $\mathbb{Q}_p$.  

Before we prove the missing result about $f(2p-1)$, we state a lemma which follows by a simple induction.

\begin{lemma}\label{lem:ugly}
Let $n\in\mathbb{N}$, then it is
\begin{enumerate}[(i)]
\item $\sum_{i=2}^{2n}(-1)^i \cdot i =1+n$
\item $\sum_{i=2}^{2n} (-1)^i \cdot \binom{i}{2} = n^2$
\end{enumerate}
\end{lemma}
 
\begin{lemma}\label{mainlemma}
It is
\begin{enumerate}[(i)]
\item $f(2p-1) \equiv 0 \mod{p^3}$
\item $f'(2p-1) \not\equiv 0 \mod{p^2}$
\end{enumerate}
\end{lemma}  
\begin{proof}
We rewrite $\sum_{i=2}^{\frac{p-1}{2}}(-1)^i\cdot (x^i + x^{p+1-i})=\sum_{i=2}^{p-1}(-1)^i x^i + (-1)^{\nicefrac{p-1}{2}}x^{\nicefrac{p+1}{2}}$ in \eqref{eq:poly}.
In order to prove the first congruence, we consider $f(2p-1)$ modulo $p^3$. 
Then we have
\begin{align*}
f(2p-1) &\equiv p (1-(p+1)2p) - \frac{p+1}{2}(-1 + 2p^2) -\frac{p+1}{2}\cdot(2p-1) +p \\&-\sum_{i=2}^{p-1} ((-1)^i + (-1)^{i+1} \binom{i}{1}2p + (-1)^i \binom{i}{2}4p^2 ) \\ &+ p\sum_{i=2}^{p-1}(-1)^i((-1)^i +(-1)^{i+1} i 2p ) +(-1)^{\frac{p-1}{2}}p((-1)^{\frac{p+1}{2}} + (-1)^{\frac{p-1}{2}}2p\frac{p+1}{2})
\end{align*}
\begin{align*}
\phantom{f(2p-1)} &\equiv  p-3p^2 +1 -\sum_{i=2}^{p-1} ((-1)^i + (-1)^{i+1} i2p + (-1)^i \binom{i}{2}4p^2 ) \\ &+ p\sum_{i=2}^{p-1}(-1)^i((-1)^i +(-1)^{i+1} i 2p )\\
 & \equiv  p-3p^2 +1 -\sum_{i=2}^{p-1} (-1)^i -\sum_{i=2}^{p-1}(-1)^{i+1} i2p -\sum_{i=2}^{p-1}(-1)^i \binom{i}{2}4p^2 \\ & + p(p-2)-2p^2(\frac{p(p-1)}{2}-1) \\
&\equiv -p + 2p\sum_{i=2}^{p-1}(-1)^i i - 4p^2 \sum_{i=2}^{p-1} (-1)^i \binom{i}{2} \mod{p^3}
\end{align*}
Applying Lemma \ref{lem:ugly} yields $ f(2p-1)\equiv 0 \mod{p^3}$, as claimed.

For the second congruence, we use
\[
f'(x)=p(p+1)x^p -p\frac{p+1}{2}x^{p-1} - \frac{p-1}{2} - \sum_{i=2}^{p-1} i x^{i-1}+p\sum_{i=2}^{p-1}(-1)^i i x^{i-1} + (-1)^{\frac{p-1}{2}}p\frac{p+1}{2}x^{\frac{p-1}{2}}.
\]
Modulo $p^2$ we have
\begin{align*}
f'(2p-1) &\equiv p(-1+2p^2) - p \frac{p+1}{2} (1-(p-1)2p) - \frac{p-1}{2} - \sum_{i=2}^{p-1} i ((-1)^{i-1} + (-1)^i (i-1)2p) \\& +p\sum_{i=2}^{p-1}(-1)^i i(-1)^{i-1} + (-1)^{\frac{p-1}{2}} p \frac{p+1}{2}(-1)^{\frac{p-1}{2}} \\
&\equiv \frac{-3p-1}{2} - \sum_{i=2}^{p-1} (-1)^{i-1}i -4p\sum_{i=2}^{p-1}(-1)^i \binom{i}{2}- p\sum_{i=2}^{p-1}i \\
&\overset{\text{Lemma }\ref{lem:ugly}}{\equiv} \frac{-3p-1}{2} +1 +\frac{p-1}{2} -4p \left(\frac{p-1}{2} \right)^2 \equiv -2p \mod{p^2}
\end{align*}
Since $p$ is odd, this shows $f'(2p-1)\not\equiv 0 \mod{p^2}$.
\end{proof}

\begin{prop}\label{prop:nocycl}
The polynomial $f$ from \eqref{eq:poly} has a cyclotomic factor if and only if $p\equiv 1 \mod{12}$. In this case, the cyclotomic factor is $x^2 +x +1$.
\end{prop}
\begin{proof}
We already know that $f$ splits completely over $\mathbb{Q}_p$. Hence, if some root of unity $\zeta$ is a root of $f$, then $\zeta$ is a $(p-1)$th root of unity. Let $\zeta$ be a $(p-1)$th root of unity. Then
\begin{align}\label{eq:rootof1}
f(\zeta) &= p\zeta^{p+1} -\frac{p+1}{2}\zeta^p - \sum_{i=2}^{p-1}\zeta^i + p \sum_{i=2}^{p-1}(-\zeta)^i + (-1)^{\frac{p-1}{2}}p\zeta^{\frac{p+1}{2}}-\frac{p+1}{2}\zeta +p \nonumber \\
&= p (\zeta^2 +1 +(-1)^{\frac{p-1}{2}}\zeta^{\frac{p+1}{2}})
\end{align}
This is zero if and only if $(-1)^{\frac{p-1}{2}}\zeta^{\frac{p+1}{2}}=\zeta$ is a third root of unity. This equation is true if and only if $2\mid \frac{p-1}{2}$ and $3\mid \frac{p-1}{2}$. Hence, if and only if $p\equiv 1 \mod{12}$.
\end{proof}

\begin{cor}
For all primes $p\in S=\{3,5,7,11,17,19,23,29,31,41,47\}$ it is $n_p< \frac{\log(p)}{p}$. Moreover, for $p\in\{3,5,7\}$ it is $n_p \leq \frac{\log(p)}{p+1}$.
\end{cor}
\begin{proof}
Let $p$ be a prime such that $p\not\equiv 1\mod{12}$. Some root $\alpha$ of $f$ must satisfy $h(\alpha)\leq \frac{1}{p+1}\log(M(f))$. Moreover, by Proposition \ref{prop:nocycl}, $\alpha$ is not a root of unity and hence $n_p \leq \frac{1}{p+1}\log(M(f))$. It remains to calculate $M(f)$. 

We check with a computer $\frac{\log(M(f))}{p+1} < \frac{\log(p)}{p}$ for all $p\in S$. For $p\in\{3,5,7\}$ all roots of $f$ lie on the unit circle, and hence $M(f)=p$, which proves the second claim.
\end{proof}

\end{document}